\documentclass[11pt]{article}
\usepackage{amsmath,amssymb,latexsym}
\newcommand{\EQ}{\begin{equation}}
\newcommand{\EN}{\end{equation}}
\newtheorem{theorem}{Theorem}

\newtheorem{corollary}{Corollary}
\newtheorem{proposition}{Proposition}
\newtheorem{lemma}{Lemma}

\newtheorem{ex}{Example}
\newtheorem{remark}{Remark}

\newcommand{\F}{{\mathbb F}_{q^n}}

\newcommand{\FB}{{\mathbb F}_q}

\newcommand{\al}{{\alpha}}

\newenvironment{proof}{\begin{trivlist}\item[]{\em Proof. }}%
{\samepage\hfill$\diamond$\end{trivlist}}
\title{Constructing permutations of finite fields via linear translators}
\author{ Gohar M. Kyureghyan}

\date{Department of Mathematics,
Otto-von-Guericke-University Magdeburg,\\
D-39016 Magdeburg, Germany}
\begin{document}
\maketitle
\begin{abstract}
We study the permutations of the finite field $\F$ given by 
$x + \gamma\, f(x)$, where $\gamma \in \F$ is a linear translator of  $f:\F \to \FB$. 
We determine the cycle structure and the inverse of such a permutation.
We describe several families of permutation polynomials obtained using functions with linear translators.
\end{abstract}


\section{Introduction}

Let $q$ be a power of a prime number and $\F$ be the finite field  of order $q^n$. 
Any polynomial $F(X) \in \F[X]$ defines a mapping
$
F :\F \to  \F \mbox{ via } x \mapsto  F(x),
$
which is called the associated mapping of $F(X)$.  Furthermore, any mapping from
 a finite field into itself is given by a polynomial.
A polynomial
$F(X)$ is called a {\em permutation polynomial}
of $\F$ if its associated mapping  is a permutation. Permutation polynomials over finite fields have  a  variety  of  theoretical and practical applications. 
Permutations described by  ``nice'', for instance  sparse, polynomials are of special interest.

A permutation $F$, for which the mapping $G(x) = F(x) + x$ is a permutation as well, is called a {\em complete mapping}, while
the mapping $G(x)$  in its turn is called an {\em orthomorphism}. Orthomorphisms yield  latin squares which are orthogonal
to the Caley table of the additive group of $\F$.
More generally, let $H : \F \to \F$ and define
$$
\mathcal{M}(H) := \{  c \in \F ~|~ H(x) + cx \mbox{ is a permutation of } \F \}.
$$
Clearly, $F$ is a  complete mapping if and only if $\{0,1 \} \subseteq \mathcal{M}(F)$, and 
 $G$ is an  orthomorphism if and only if $\{0,-1 \} \subseteq \mathcal{M}(G)$.
The set $D_H$ of directions determined by the mapping $H$ is defined as follows
$$
D_H := \left\{ \frac{H(x) - H(y)}{x-y} ~|~ x,y \in \F, x\ne y \right\}.
$$
Note that $c \in \mathcal{M}(H)$ if and only if $-c \notin D_H$, and hence
a mapping $H$ with a large set $\mathcal{M}(H)$ determines a small number of directions, or equivalently a small blocking set of R\'edei type in $PG(2, q^n)$.
Sharp bounds on the size of $\mathcal{M}(H)$ are proved in \cite{ball, evans-greene-niederreiter, redei}. Further references on this
topic may be found in \cite{ball}.

Let $\gamma \in \F$ and  $G:\F \to \F, ~f:\F \to \FB$ and define
\begin{equation}\label{pp-form}
F(x) = G(x) + \gamma \,f(x).
\end{equation}  
The description of the mappings $F$ and $G$ with respect to a  basis of $\F$ over $\FB$ containing the element  $\gamma \ne 0$ yields
a geometrical relation between them.
Indeed, let $(\gamma, \beta_1, \ldots , \beta_{n-1})$ be a basis of  $\F$ over $\FB$. Denote by 
$g_i:\F \to \FB, ~0\leq i \leq n-1,$ the coordinate functions of $G$ with respect to the chosen basis. Then for any $x \in \F$ it holds
$$
G(x) = g_0(x)\gamma + g_1(x)\beta_1 + \ldots +  g_{n-1}(x)\beta_{n-1}
$$
and
$$
F(x) = (g_0+f)(x)\gamma + g_1(x)\beta_1 + \ldots +  g(x)\beta_{n-1}.
$$
Hence $F$ is obtained from $G$ by changing  its $\gamma$-coordinate function. In \cite{edel-pott} APN mappings and 
in  \cite{charpin-kyureg-bpp, charpin-kyureg-gpp, marcos} permutations of
form (\ref{pp-form}) are studied.

This paper continues the study of permutations of form (\ref{pp-form}).
In Section \ref{sec-inv} we briefly introduce a concept of a linear translator.
Further we characterize permutations $x + \gamma\, f(x)$ of $\F$, where
$\gamma \in \F$ is a linear translator of  $f:\F \to \FB$.
 In particular, we
determine explicitly the inverse and the cycle structure of such a permutation. 
In Section \ref{sec-pp} we describe explicit constructions of permutation polynomials
based on the results of Section \ref{sec-inv}.
Moreover, we show that the knowledge of the inverse mapping of a permutation $x\mapsto x + \gamma\, f(x)$
allows to construct permutations by changing two coordinate functions of the identity mapping.
Finally, in Section \ref{sec-gen} we give further constructions  of permutation polynomials of the shape
$L(X) + \gamma \,f(X)$,  where $L(X)$ is a linearized polynomial and
 $\gamma \in \F$ and  the polynomial $f(X) \in \F[X]$ induces a mapping into $\FB$.
In the constructions of Section \ref{sec-gen} the element $\gamma$ is not necessarily a linear translator of $f$.

\section{Remarks on permutations $x + \gamma\, f(x)$} \label{sec-inv}

Let $f$ be a  mapping from a finite field $\F$ into its 
 subfield $\FB$.
Such a mapping can be represented by $Tr(G(x))$ for some
 (not unique) mapping 
$G:\F \to \F$, where $Tr$ is the trace mapping from $\F$ onto $\FB$ given by
\[
Tr(y)=y + y^{q} +  y^{q^2}+  \dots + y^{q^{n-1}}.
\]
Indeed, we need just to choose the value of $G(x)$ to satisfy $Tr(G(x)) = f(x)$.

A non-zero element $\al \in \F$ is called
an $a$-linear translator (or linear structure, cf.~\cite{charpin-kyureg-bpp, charpin-kyureg-gpp}) for the mapping $f: \F \to \FB$ if
\EQ\label{eq:ls}
 f(x+ u\al) - f(x) = ua, 
\EN
for all  $x\in\F, u \in \FB$ and some fixed $ a \in \FB$.
Note that if (\ref{eq:ls}) is satisfied then $a$ is uniquely determined, more exactly, it holds $a = f(\alpha) - f(0)$.
 The next two results are proved in \cite{lai}, see also \cite{charpin-kyureg-gpp}. 
\begin{proposition}\label{linsp}
Let $\al,\beta \in \F^*,~\al + \beta \ne 0$ and $a,b \in \FB$. If
$\al$  is an $a$-linear translator and $\beta$ is a $b$-linear translator of a
 mapping $f: \F \to \FB$ , then
$$
\al + \beta \textrm{ is an } (a+b)-\textrm{linear translator of } f
$$
and for any $ c \in \FB^*$
$$
c\cdot \al  \textrm{ is a } (c\cdot a)-\textrm{linear translator of $f$}.
$$
In particular, if $\Lambda^*(f)$ denotes the set of all linear translators of 
$f$, then $\Lambda(f)=\Lambda^*(f)\cup\{0\}$ is a $\FB$-linear subspace of $\F$.
\end{proposition}

In particular Proposition \ref{linsp} shows that the restriction of the mapping  $f(x)-f(0)$ on 
the subspace $\Lambda(f)$ is a $\FB$-linear mapping. So $\Lambda(f)=\F$ if and only if $f(x)$ is an
affine mapping, or equivalently if $f(x)= Tr(\beta x) +b$ for some $\beta \in \F$ and $b \in \FB$.
Moreover, the following theorem holds:

\begin{theorem}\label{lai}
Let $G:\F \to \F$ and $f(x)=Tr(G(x))$.
Then $f$   has a linear translator if and only if there is
a non-bijective $\FB$-linear mapping $L:\F \to \F$ such that
\EQ\label{eq-lai}
f(x)=Tr(G(x)) = Tr\big(H\circ L(x)+ \beta x \big)
\EN
for some $H:\F\to \F$ and  $\beta \in \F$.
In this case,  the kernel of $L$ is contained in the subspace $ \Lambda(f)$.
\end{theorem}

By Theorem \ref{lai}  any $\FB$-linear mapping with a known kernel allows to construct
a mapping with known linear translators.
The following result  is an example of this.

\begin{lemma}\label{lemma-lai}
Let $H:\F \to \F$ be an arbitrary mapping, $\gamma, \beta  \in \F,~ 
\gamma \ne 0$. Then
$\gamma$ is a $Tr\big(\beta \gamma \big) $-linear translator of $f(x)= Tr(G(x))$ where
$$
G(x) = H(x^q -\gamma^{q-1} x)+ \beta x.
$$
\end{lemma}
\begin{proof}
Indeed, for any $u \in \FB$ it holds
\begin{eqnarray*}
f(x+u\gamma ) & = &  Tr\Big(H\big((x+u\gamma )^q -\gamma^{q-1} (x+u\gamma )\big)+ \beta (x+u\gamma )\Big)\\
 & = &  Tr\Big(H\big(x^q+u\gamma ^q -\gamma^{q-1} x-u\gamma ^q\big)+ \beta x+u\beta \gamma )\Big)\\
  & = &  Tr\Big(H\big(x^q  -\gamma^{q-1} x \big) + \beta x \Big) + u  Tr(\beta \gamma )\\
  & = & f(x) + u  Tr(\beta \gamma ).
\end{eqnarray*}
\end{proof}
Another family of mappings with known linear translators is given in the next lemma, which can be verified
by direct calculations similarly to Lemma \ref{lemma-lai}. 
 
\begin{lemma}\label{deriv}
Let $g: \F \to \FB$ and $\al\in\F^*$.
Then  for any $c\in\FB^*$ the element $c\al$  is a $0$-linear structure of
\[
f(x)=\sum_{u\in\FB}g(x+u\al).
\]
\end{lemma}

\begin{lemma}\label{lem:monom}
Let $n=4k$, $\beta \in \F$ and $\gamma \in  \lambda^{(q^{4k}-1)/2(q^2-1)}\mathbb{F}_{q^2}^*$
with $\lambda$ being a primitive element of $\F$.
Then $\gamma$ is a $Tr\big(\beta \gamma \big)$-linear structure of
$$
f(x) =  Tr(x^{q+1} + \beta\,x ).
$$
\end{lemma}
\begin{proof}
Note that $\gamma ^{q^2-1} = -1$ or equivalently $\gamma ^{q^2} + \gamma =0$. Taking the latter identity to the power
$q^{n-1}$, we obtain $\gamma ^q + \gamma ^{q^{n-1}} =0$. Further, since $(\gamma ^{q+1})^{q-1} = -1$, it holds
$(\gamma ^{q+1})^{q} + \gamma ^{q+1} =0$, and consequently $Tr(\gamma ^{q+1})=0$.
Using this properties of $\gamma$, for any $x \in \F$ and  $u \in \FB$ we obtain
\begin{eqnarray*}
f(x + u \gamma) &  = &   Tr((x + u\gamma)^{q+1} + \beta x + \beta \gamma u ) \\ & = & Tr( x^{q+1}   + \gamma ^{q}ux +\gamma u x^{q} + \gamma ^{q+1}u^2 + \beta x+ u\beta\gamma ) \\
               & = & f(x) + u\,Tr((\gamma ^q + \gamma ^{q^{n-1}} )x) + u^2  Tr(\gamma^{q+1}) + u Tr(\gamma\beta) \\
               & = &  f(x)  + u Tr(\gamma\beta).
\end{eqnarray*}
\end{proof}

Lemmas \ref{lemma-lai}, \ref{deriv}, as well as the Theorem \ref{pp-linstr}, are straightforward generalizations of  results given in \cite{charpin-kyureg-gpp} 
in the case of  prime  $q$. The proof of Theorem \ref{pp-linstr} differs from the one given in \cite{charpin-kyureg-gpp}.

\begin{theorem}\label{pp-linstr}
Let   $\gamma \in \F$ be a 
$b$-linear translator of $f :\F \to \FB$. 
\begin{itemize}
\item[(a)]
Then $F(x) = x + \gamma \,f(x)$ is
 a  permutation of $\F$ if $b\ne -1$.
\item[(b)] Then 
 $F(x) = x + \gamma \,f(x)$ is
 a  $q-\,to\,-1$ mapping of $\F$  if $b = -1$.
\end{itemize}
\end{theorem}
\begin{proof}
Let $x, y \in \F$  satisfy $F(x) = F(y)$. Then
\begin{equation}\label{eq-value}
F(x) = x + \gamma \,f(x) = y + \gamma \,f(y) = F(y),
\end{equation}
and hence
$$
x = y + \gamma \,\left(f(y)-f(x)\right)  = y + \gamma a,
$$
where $a =  f(y)-f(x)\in \FB$. Using the definition of a linear translator we get
$$
a = f(y)-f(x) = - \left(f(y + \gamma a ) - f(y)\right) = -ab.
$$
If $b \ne -1$, then the last equality implies $a=0$ and hence $f(y) = f(x)$. Finally, (\ref{eq-value}) forces $x = y$,
which proves (a). Suppose $b = -1$. Then the above arguments show that $F(x) = F(y)$ only if $x =  y + \gamma a$ for some $a \in \FB$.
To complete the proof of (b) it remains to see that
$$
F(y + \gamma a) = y + \gamma a + \gamma f(y + \gamma a) =  y + \gamma a + \gamma f(y) - \gamma a = F(y)
$$
for any $a \in \FB$.
\end{proof}

If we  choose $f(x) = Tr(x)$, then Theorem \ref{pp-linstr} states that the mapping
$
x \mapsto x + \gamma  \,Tr(x)
$
is a permutation of $\F$ if and only if $Tr(\gamma) \ne -1$. Consequently,
the mapping 
$x \mapsto  Tr(x) + \delta\, x $
is a permutation of $\F$ if and only if $\delta \ne 0$ and $Tr(\delta^{-1}) \ne -1$, and thus 
$$
\mathcal{M}(Tr) = \{ \delta \in \F^* ~|~ Tr(\delta^{-1}) \ne -1 \}.
$$
In particular $|\mathcal{M}(Tr)| = q^n - q^{n-1} - 1$. The mapping $Tr(x)$ was mentioned  in \cite{redei} to show 
that certain bounds on $|\mathcal{M}(\cdot)|$ are tight, see \cite{ball} for further details.

Let us consider an arbitrary $f :\F \to \FB$.
The mapping $f(x) + \delta x$ is a permutation of $\F$ if and only if $\delta \ne 0$ and $x +  \delta^{-1}f(x)$ is a permutation.
So by Theorem  \ref{pp-linstr}, the inverse of any $b$-linear translator $\delta^{-1}$  of $f$  with $b = f(\delta^{-1}) - f(0) \ne -1$
is contained in $\mathcal{M}(f)$. So it holds
\begin{equation}\label{eq:dir}
\{\delta \in \F^* ~|~ \delta^{-1} \in \Lambda^*(f)~ \mbox{ and }~ f(\delta^{-1}) - f(0) \ne -1 \}  \subseteq \mathcal{M}(f).
\end{equation}
This shows that the mappings $f : \F \to \FB$
with many linear translators determine few directions. The next result 
describes such mappings.

\begin{theorem}\label{th:dir}
Let $g: \FB \to \FB$ be such that $g(0) =0$ and $-1 \notin \{ g(y) ~|~ y \in \FB\}$. 
Given a non-zero $\alpha \in \F$ define $h: \F \to \FB$ by $h(x) = g(Tr(\alpha x))$ for any $x \in \F$. 
Then $ \{\delta \in \F^* ~|~ Tr( \alpha \delta ^{-1}) =0 \} \subseteq \mathcal{M}(h)$ and hence $|\mathcal{M}(h)| \geq q^{n-1}-1$.
\end{theorem}
\begin{proof}
Theorem \ref{lai} implies that $ \{ y \in \F^* ~|~ Tr(\alpha y) = 0\} \subseteq \Lambda^*(h)$.
Then from (\ref{eq:dir}) it follows
$$
\{\delta \in \F^* ~|~  g(Tr(\alpha \delta^{-1})) \ne -1 \mbox{ and } Tr(\alpha \delta ^{-1}) =0 \} \subseteq \mathcal{M}(h). 
$$
It remains to note that 
$$
\{\delta \in \F^* ~|~ g(Tr(\alpha \delta^{-1})) \ne -1 \mbox{ and } Tr(\alpha \delta ^{-1}) =0 \}  =  
\{\delta \in \F^* ~|~   Tr(\alpha \delta ^{-1}) =0 \}
$$
since by the choice of $g$ the element $-1$ does not belong to its image set.
\end{proof}

Remark that if in Theorem \ref{th:dir} the mapping $g$  is not  affine on $\FB$, then 
$h$ is not affine on $\F$. Hence $ \{ y \in \F^* ~|~ Tr(\alpha y) = 0\} = \Lambda^*(h)$
for such mappings using Proposition \ref{linsp}.
Next we give an explicit example of such a mapping $h$ if $q$ is odd.

\begin{ex}
Let $q$ be odd and $q-1 = 2^i\cdot d$ with $d$ odd. Let $h: \F \to \FB$ be defined by
$$
h: x \mapsto \Big(Tr(x)\Big)^{2^{i}} ~\mbox{ for any }~x \in \F.
$$
Then $ \{\delta \in \F^* ~|~ Tr( \delta ^{-1}) =0 \} \subseteq \mathcal{M}(h)$.
\end{ex}

It is obvious that the permutations described in Theorem \ref{pp-linstr} are never orthomorphisms, and hence
never complete if $q$ is even. Corollary \ref{complete} characterizes all such 
complete mappings.

\begin{corollary}\label{complete}
Let $q$ be odd and  $\gamma \in \F$ be a 
$b$-linear translator of $f :\F \to \FB$. Then 
 $F(x) = x + \gamma \,f(x)$ is
 a  complete mapping of $\F$ if and only if $b \notin \{-1, -2\}$.
\end{corollary}
\begin{proof}
Indeed $F$ is a permutation if and only if $b \ne -1$ by Theorem \ref{pp-linstr}.
Consider $F(x)+x = 2x + \gamma\, f(x)$. The latter is a permutation of $\F$
if and only if $x + \frac{\gamma}{2}\,f(x)$ is so. Proposition \ref{linsp} shows that
${\gamma}/{2}$ is a $b/2$-linear translator of $f$. Thus $F(x) + x$ is a permutation of $\F$
if and only if $b \ne -2$, completing the proof.
\end{proof}

Our next goal is to determine the cycle structure and the inverse of a permutation described in Theorem \ref{pp-linstr}.
For an integer $k\geq 1$, define 
$$
F_k(x) = \underbrace{F\circ F \circ \ldots \circ F}_{k \mbox{ {\tiny times}}}(x)
$$
to be the $k$-fold  composition of the mapping $F$ with itself.

\begin{lemma}\label{k-comp}
Let $\gamma \in \F$ be a 
$b$-linear translator of $f:\F \to \FB$ and
 $F(x) = x + \gamma \,f(x)$. Then  for any $k \geq 1$ it holds
$$
F_k(x) = x +  B_k \,\gamma \,f(x),
$$
where
\begin{equation}\label{eq:B}
B_k = 1 + (b+1) + \ldots + (b+1)^{k-1} = \left\{ \begin{array}{ll} k & \mbox{ if } b =0 \\ 
                                                               \frac{(b+1)^k-1}{b} & \mbox{ if } b \ne 0.
\end{array} \right.
\end{equation}
\end{lemma}
\begin{proof}
Our proof is by induction on $k$. Clearly it holds for $k=1$.
For $k \geq 2$ we have
\begin{eqnarray*}
F_k(x) & = & F \circ F_{k-1}(x)  =  \left( x + \gamma \,f(x)\right) \circ \left(x + B_{k-1} \,\gamma \,f(x)\right) \\
& = & x + B_{k-1} \,\gamma \,f(x) + \gamma \,f(x + B_{k-1} \,\gamma \,f(x)).
\end{eqnarray*}
Since $\gamma$ is a $b$-linear translator for $f$ and   $B_{k-1} \, f(x) \in \FB$, it holds
$$
f(x + B_{k-1} \,\gamma \,f(x)) = f(x) + B_{k-1}  \,f(x) \, b.
$$
Then we get
$$
F_k(x) = x + (B_{k-1}  + 1 +   B_{k-1}\,b ) \gamma \,f(x)  =  x + (1 + (b+1) B_{k-1})\gamma \,f(x).
$$
It remains to note that $1 + (b+1) B_{k-1} = B_k$.
\end{proof}
As a direct consequence of Lemma \ref{k-comp}, we determine the inverse mapping and the cycle structure of the considered permutations.

\begin{theorem}\label{inverse}
Let $\gamma \in \F$ be a 
$b$-linear translator of $f : \F \to \FB$ and $b \ne -1$. Then the inverse mapping of the permutation
$F(x) = x + \gamma \,f(x)$ is $$F^{-1}(x) = x - \frac{\gamma}{b+1} \,f(x).$$
\end{theorem}
\begin{proof}
Let $b=0$. 
Then from  Lemma \ref{k-comp} it follows that $F_p(x) =x$   where $p$ is the characteristic of $\F$. Hence   
the inverse mapping of $F$ is   $F_{p-1}(x) = x - \gamma \,f(x)$.
Let  $b\ne 0$. Then again  using Lemma \ref{k-comp}  the inverse mapping of $F$ is $F_{l-1}$, where $l$ is the  order of $~b+1$ in $\FB^*$.
It remains to note that
\begin{eqnarray*}
F_{l-1}(x)& = & x + \frac{(b+1)^{l-1}-1}{b}\,\gamma \,f(x)\\ & = & x + \left(\frac{1}{b+1}-1\right)\frac{1}{b}\, \gamma \,f(x) \\ & = & x - \frac{\gamma}{b+1} \,f(x),
\end{eqnarray*}
since $(b+1)^{l-1} = (b+1)^{-1}$.
\end{proof}

Obviously, an element $u \in \F$ is a fixed point for $F(x) = x + \gamma f(x), ~\gamma \ne 0,$ if and only if $f(u)=0$.
The next theorem describes the  cycle decomposition of such permutations.

\begin{theorem}
Let $\gamma \in \F$ be a 
$b$-linear translator of $f:  \F \to \FB$ and $b \ne -1$. Consider the permutation defined by
$F(x) = x + \gamma \,f(x)$. Set 
$$N = q^n - |\{x \in \F ~|~ f(x) =0\}|.$$

\begin{itemize}
\item[(a)] If  $b=0$, then the permutation $F$ is a composition of $N/p$ disjoint cycles  of length $p$ (in the 
symmetric group $S_{_{\F}}$), where $p$ is the characteristic of $\F$. 
Moreover, an element $u \in \F$ with $f(u) \ne 0$ is contained in the cycle $(u_0, u_1, \ldots, u_{p-1})$, where $u_k = u + k\,\gamma f(u)$.
\item[(b)] If $b\ne 0$, then  the permutation $F$ is a  composition of $N/l$ disjoint cycles  of length $l$, where $l$ is the  order of $(b+1)$ in $\FB^*$.
Moreover, an element $u \in \F$ with $f(u) \ne 0$ is contained in the cycle $(u_0, u_1, \ldots, u_{l-1})$, where $u_k = u + B_k\,\gamma f(u)$
and $B_k$ is defined by (\ref{eq:B}).
\end{itemize}
\end{theorem}
\begin{proof}
The proof follows from  Lemma \ref{k-comp}.
\end{proof}

\begin{remark}
A particular case of Corollary \ref{complete} and Theorem \ref{inverse} for $b=0$ are proved in \cite{marcos} for
 permutations $x+h(Tr(x))$, where $h:\FB\to \FB$ and $q$ is a prime number. In \cite{marcos} and \cite{zieve}
further permutations of $\F$ involving additive  mappings  are constructed.  
\end{remark}

\section{Families of permutation polynomials} \label{sec-pp}

In this section we demonstrate several applications of  Theorems \ref{pp-linstr} and \ref{inverse} to obtain explicit constructions of permutation polynomials.
Firstly, observe that combining  Theorem  \ref{pp-linstr} and Lemma \ref{lemma-lai} we obtain:

\begin{theorem}
Let $H(X) \in \F[X],~\gamma, \beta  \in \F$. Then
$$
F(X) = X + \gamma\, Tr\big(H(X^q -\gamma^{q-1} X)+ \beta X\big)
$$
is a permutation polynomial of $\F$ if and only if $Tr(\gamma \beta) \ne -1$.

\end{theorem}
Further families of permutation polynomials may be obtained using the following extension
of  Theorem  \ref{pp-linstr}.

\begin{theorem}\label{pp-lpp}
Let $L:\F \to \F$ be an $\FB$-linear permutation of $\F$.
Further, suppose $\gamma \in \F$  is a $b$-linear translator of $f:\F \to \FB$.
\begin{itemize}
\item[(a)]
Then $F(x) = L(x) + L(\gamma) \,f(x)$ is
 a  permutation of $\F$ if $b\ne -1$.
\item[(b)] Then 
 $F(x) = L(x) + L(\gamma) \,f(x)$ is
 a  $q-\,to\,-1$ mapping of $\F$  if $b = -1$.
\end{itemize}
\end{theorem}
\begin{proof}
Note that the mapping $F$ is the composition of $L$ and the mapping $x + \gamma  \,f(x)$. Indeed,
$$
L \big(x + \gamma \,f(x)\big) = L(x) + L\big(\gamma \,f(x)\big) = L(x) + f(x)L(\gamma).
$$
The rest follows from  Theorem  \ref{pp-linstr}.

\end{proof}
Recall that $\FB$-linear mappings of $\F$ are described by the polynomials $\sum_{i=0}^{n-1}\alpha _iX^{q^i}$ $ \in \F[X]$,
which are called  $q$-polynomials.
Hence given a  permutation $q$-polynomial, Theorem \ref{pp-lpp} combined with Lemma \ref{lemma-lai}
or Lemma \ref{deriv} yields variety of families of permutation polynomials.
As an example we consider
$L(X) = X^q + X$, which   is a permutation polynomial of $\F$
when $n$ is odd. The inverse mapping of $L$ is given by
$$
L^{-1}(X) = X^{q^{n-1}} - X^{q^{n-2}} + \ldots + X^{q^2} - X^q +X.
$$
Using these polynomials and Theorem  \ref{pp-lpp}, Lemma \ref{lemma-lai} we obtain:
\begin{theorem}
Let $H(X) \in \F[X],~\gamma, \beta  \in \F$ and $n$ be odd. 
\begin{itemize}
\item[(a)]
Then
$$
X^q + X + (\gamma^q + \gamma)\, Tr\big(H(X^q -\gamma^{q-1} X)+ \beta X\big)
$$
is a permutation polynomial of $\F$ if and only if $Tr(\gamma \beta) \ne -1$.
\item[(b)] Then 
$$
\sum_{i=1}^n(-1)^{i+1}X^{q^{n-i}}  + \Big(\sum_{i=1}^n(-1)^{i+1}\gamma^{q^{n-i}}\Big)\, Tr\big(H(X^q -\gamma^{q-1} X)+ \beta X\big)
$$
is a permutation polynomial of $\F$ if and only if $Tr(\gamma \beta) \ne -1$.
\end{itemize}
\end{theorem}

The following result is of interest if $\gamma$ and $ \delta$ are linearly independent
over $\FB$,  otherwise it is covered by Theorem \ref{pp-linstr}. 
It describes permutations of $\F$ obtained form the identity mapping by changing its
$\gamma-$ and $\delta-$coordinate functions.
\begin{theorem}\label{pp-2linstr}
Let   $\gamma, \delta \in \F$. Suppose $\gamma$ is a 
$b_1$-linear translator of $f :\F \to \FB$ and a $b_2$-linear translator of $g :\F \to \FB$, and moreover
$\delta$ is a 
$d_1$-linear translator of $f$ and a $d_2$-linear translator of $g$.
Then
$$
F(x) = x + \gamma \,f(x) + \delta\,g(x)
$$
is a  permutation of $\F$, if $b_1 \ne -1$ and $ d_2 -  \frac{d_1b_2}{b_1+1} \ne -1$.
\end{theorem}

\begin{proof}
Since $b_1 \ne -1$, the mapping $G(x) =  x + \gamma \,f(x)$ is a permutation by Theorem \ref{pp-linstr}.
Then using Theorem \ref{inverse}, the inverse mapping of $G$ is 
$$G^{-1}(x) = x - \frac{\gamma}{b_1+1}  \,f(x).$$
Consider
\begin{eqnarray*}
F\circ G^{-1}(x) &=& G \circ G^{-1}(x) + \delta \, g( x - \frac{\gamma}{b_1+1}  \,f(x))\\
& = & x + \delta\left( g(x) - \frac{b_2}{b_1+1}f(x) \right)\\
& = & x + \delta\, h(x).
\end{eqnarray*}
Note that $\delta$ is a $\left(d_2 -  \frac{d_1b_2}{b_1+1}\right)$-linear translator of $h:\F \to \FB$. Indeed, for any $u \in \FB$ it holds
\begin{eqnarray*}
h(x + \delta u) & = & g(x + \delta u) - \frac{b_2}{b_1+1}f(x +\delta u) \\ &=& g(x) + d_2u - \frac{b_2}{b_1+1}(f(x) + d_1u)\\
& = & h(x) + \left(d_2 -  \frac{d_1b_2}{b_1+1}\right)u.
\end{eqnarray*}
Theorem \ref{pp-linstr} completes the proof.
\end{proof}

As an application of Theorem \ref{pp-2linstr} we obtain:

\begin{theorem}
Let $\alpha \in \F\setminus \FB$ and 
$$
M(X) = X^{q^2} -(1 + (\alpha^q - \alpha)^{q-1})X^q + (\alpha^q - \alpha)^{q-1}X.
$$
Let $H_1, H_2 :\F \to \F$ and $\beta_1, \beta_2 \in \F$ be arbitrary. Then
$$
F(X) = X + Tr\big(H_1(M(X))+ \beta_1 X\big) + \alpha\, Tr\big(H_2(M(X))+ \beta_2 X\big)
$$
is a permutation polynomial of $\F$ if  
\begin{itemize}
\item
$Tr(\beta_1) \ne -1$ and $Tr(\alpha \beta_2) -  \frac{Tr(\alpha \beta_1)Tr(\beta_2)}{Tr(\beta_1)+1} \ne -1 $; or
\item
 $Tr(\alpha \beta_2) \ne -1$ and $Tr(\beta_1) -  \frac{Tr(\beta_2)Tr(\alpha \beta_1)}{Tr(\alpha \beta_2)+1} \ne -1$.
\end{itemize}
\end{theorem}

\begin{proof}
We are in the setting of Theorem \ref{pp-2linstr}: In the first case, $\gamma = 1,~\delta = \alpha$ and
$f(x) = Tr\big(H_1(M(x))+ \beta_1 x\big),~ g(x) = Tr\big(H_2(M(X))+ \beta_2 X\big)$.
Direct calculations show that $1$ is a $Tr(\beta_1)$-linear translator of $f$ and
is a  $Tr(\beta_2)$-linear translator of $g$, since $M(1)=0$. Similarly, $\delta$ is
a $Tr(\delta\beta_1)$-linear translator of $f$ and
is a  $Tr(\delta \beta_2)$-linear translator of $g$. In the second case the roles of $f$ and $g$ are exchanged.

\end{proof}

\section{Further constructions} \label{sec-gen}

The results of this section are inspired by Theorem 1 from \cite{marcos} and meanwhile generalize
it and most of the constructions of permutations from \cite{charpin-kyureg-gpp}.

\begin{theorem}\label{th:general}
Let $\gamma \in \F$  be a $b$-linear translator of  $f: \F \to \FB$ and $h:\FB \to \FB$.
Define $F:\F \to \F$  by
$$
F(x) = x + \gamma\, h(f(x)).
$$
 Then $F$  permutes $\F$ if and only if the mapping $g(u)= b h(u) + u$ permutes $\FB$.
\end{theorem}
\begin{proof}
The arguments from the proof of Theorem \ref{pp-linstr} show  that if for some $x,y \in \F$ it holds $F(x) = F(y)$, 
then $x = y + \gamma a$ with $a \in \FB$. 
Further, $F(y) = F(y+ \gamma a)$ implies
\begin{equation}\label{eq:x-u}
y + \gamma a + \gamma\, h(f(y) + b a) = y + \gamma\, h(f(y)),
\end{equation}
since
\begin{eqnarray*}
F(y + \gamma a)  =  y + \gamma a + \gamma\, h(f(y + \gamma a)) 
  =  y + \gamma a + \gamma\, h(f(y) + b a).
\end{eqnarray*}
Equation (\ref{eq:x-u}) is equivalent to
\begin{equation}\label{eq:x-u-r}
a +  h(f(y) + b a) =   h(f(y)).
\end{equation}
If $b=0$  then from (\ref{eq:x-u-r}) forces  $a=0$ and hence  the statement is true for that case.
If $b \ne 0$, then  (\ref{eq:x-u-r})  can be reduced to
$$
 h(f(y) + b a) + b^{-1} \big( f(y) + b a) =   h(f(y)) + b^{-1}f(y),
$$
and hence $g(f(y) + b a) = g(f(y))$. The latter equation is satisfied only for $a=0$ if and only if $g$ is a
permutation of $\FB$.

\end{proof}
Observe that Theorem \ref{pp-linstr} follows from Theorem \ref{th:general} if we take $g$ to be
the identity mapping. Next family of permutation polynomials is obtained combining Theorem  \ref{th:general} and Lemma \ref{lem:monom}.

\begin{theorem}
Let $n=4k$, $\beta \in \F$ and  $\gamma \in  \lambda^{(q^{4k}-1)/2(q^2-1)}\mathbb{F}_{q^2}^*$
with $\lambda$ being a primitive element of $\F$.
Further, suppose $t$ is a positive integer with $\gcd(t,q-1)=1$. Then the polynomial
$$
F(X) = X + \gamma \, Tr(\gamma \beta)^{q-2} \left( \left( Tr(X^{q+1} + \beta x)\right)^t - Tr(X^{q+1} + \beta x)\right)
$$
is a permutation polynomial of $\F$.
\end{theorem}
\begin{proof}
We apply Theorem \ref{th:general} with  $h(u) = Tr(\gamma \beta )^{q-2}(u^t - u)$ and $f(x) =  Tr(x^{q+1} + \beta x)$.
Lemma \ref{lem:monom} shows that $\gamma$ is a $ Tr(\gamma \beta )$-linear translator of $f$.
To complete the proof note that the mapping
$$
Tr(\gamma \beta )h(u) +u = \left\{ \begin{array}{ll}
u & \mbox{ if }~Tr(\gamma \beta )=0\\
u^t & \mbox{ otherwise, }
\end{array} \right .
$$
and thus $h$ is a permutation of $\FB$.
\end{proof}

\begin{theorem}\label{th:lin-general}
Let $L:\F \to \F$ be an $\FB$-linear permutation of $\F$.  
Let $\gamma \in \F$  be a $b$-linear translator of  $f: \F \to \FB$ and $h:\FB \to \FB$.
Then the mapping 
$$
G(x) = L(x) + L(\gamma) \, h(f(x))
$$ 
permutes $\F$ if and only if $g(u)= b h(u) + u$  permutes $\FB$.
\end{theorem}
\begin{proof}
Note that $G$ is a composition of $L$ and $F(x) =  x + \gamma\, h(f(x))$. Indeed,
$$
L\big( x + \gamma\, h(f(x)) \big)= L(x) + h(f(x))L(\gamma).
$$
The rest of the proof follows from Theorem \ref{th:general}.

\end{proof}
Next we show that  Theorem 1 of \cite{marcos} is a particular case of Theorem \ref{th:lin-general}. 

\begin{theorem}[\cite{marcos}]
Let $L(X) \in \FB[X]$ be a permutation polynomial of $\F$, $h(X) \in \FB[X]$ and $\gamma \in \F$ with $Tr(\gamma) = b$.
Then the polynomial
$$
L(X) + \gamma \,h\big(Tr(X))
$$
is a permutation polynomial of $\F$ if and only if the polynomial
$$ 
L(1)X + bh(X)
$$
is a permutation polynomial of $\FB$.
\end{theorem}
\begin{proof}
Since $L$ is a permutation of $\F$ there is a unique $\delta \in \F$ such that $L(\delta) = \gamma$.
Note that $\delta$ is a $Tr(\delta)$-linear translator of the mapping $Tr(x)$. Moreover, if $L(X) = \sum_{i=0}^{n-1}a_iX^{q^i}$,
then 
$$
Tr(\gamma) = Tr(L(\delta)) = Tr(\sum_{i=0}^{n-1}a_i\delta^{q^i}) = \sum_{i=0}^{n-1}a_i Tr(\delta) = L(1) Tr(\delta),
$$
and hence $Tr(\delta) = (L(1))^{-1}b$. The rest follows from Theorem \ref{th:lin-general}.
\end{proof}

Finally, we describe permutation polynomials obtained from $\FB$-linear mappings of $\F$ with one-dimensional kernel
via changing a coordinate function. 

\begin{theorem}\label{th:lin-qto1-general}
Let $L:\F \to \F$ be an $\FB$-linear mapping of $\F$ with kernel $\alpha\FB,~ \alpha \ne 0$.  
Suppose $\alpha$  is a $b$-linear translator of  $f: \F \to \FB$ and $h:\FB \to \FB$ is a permutation of $\FB$.
Then the mapping 
$$
G(x) = L(x) + \gamma \, h(f(x))
$$ 
permutes $\F$ if and only if $b\ne 0$ and $\gamma$ does not belong to the image set of $L$.
\end{theorem}

\begin{proof}
In the case $\gamma$ belongs to the image set of $L$, the image set of $G$ is contained in that of $L$. Hence
if $G$ is a permutation, then necessarily $\gamma$ is not in the image of $L$.
Now suppose, $\gamma$ does not belong to the image set of $L$. Let $x,y \in \F$ be such that
$G(x) = G(y)$. Then
$$
 L(x) + \gamma \, h(f(x)) = L(y) + \gamma \, h(f(y)),
$$
and consequently
\begin{equation}\label{kernel}
\gamma \Big( h(f(x)) -  h(f(y))\Big) = L(y-x).
\end{equation}
Since  $\gamma$ does not belong to the image set of $L$, equation (\ref{kernel}) is possible
if and only if $ h(f(x)) =  h(f(y))$ and $y-x$ is in the kernel of $L$. So, let $y = x + a\alpha$ with $a \in \FB$.
Then   (\ref{kernel}) is reduced to
\begin{equation}\label{last}
h(f(x)) -  h(f(x + a\alpha)) = h(f(x)) - h(f(x) + ab) =0.
\end{equation}
The only solution of (\ref{last}) is  $a=0$ if and only if $b\ne 0$ and $h$ permutes $\FB$.

\end{proof}

A particular case of Theorem \ref{th:lin-qto1-general}, where the mapping $h$ is the identity mapping, is
proved in \cite{charpin-kyureg-gpp}. As an application of  Theorem \ref{th:lin-qto1-general} we describe a family of
permutation polynomials.

\begin{theorem}
Let $t$ be a positive integer with $\gcd(t,q-1)=1$, $H(X) \in \F[X]$ and $\gamma, \beta \in \F$. Then 
$$
G(X) = X^q - X + \gamma \Big(Tr(H(X^q-X) + \beta X)\Big)^t \in \F[X]
$$
is a permutation polynomial of $\F$ if and only if $Tr(\gamma) \ne 0$ and $Tr(\beta) \ne 0$.
\end{theorem}
\begin{proof}
We apply  Theorem \ref{th:lin-qto1-general} with
$L(x) = x^q-x, f(x) = Tr(H(x^q-x) + \beta x)$ and $h(u)=u^t$.
 The mapping $L(x)=x^q-x$ is  $\FB$-linear with kernel $\FB$ and
so $\alpha$ may be chosen to be $1$. The image set of $L$ consist of all elements $y$ from $\F$ with
$Tr(y)=0$ by Hilbert's Theorem 90. Further $\alpha = 1$ is a $Tr(\beta)$-linear translator of the mapping
$f(x)$. Indeed, for any $u \in \FB$ it holds
\begin{eqnarray*}
f(x+u)& = & Tr(H((x+u)^q-(x+u)) + \beta (x+u))\\ & = & Tr(H(x^q-x) + \beta x) + u \,Tr(\beta)\\ & = & f(x) + u \,Tr(\beta).
\end{eqnarray*}
It remains to note that $h(u) = u^t$ is a permutation of $\FB$ by the choice of $t$.
\end{proof}

\begin{center}
{\bf Acknowledgments}
\end{center}
The author   thanks Pascale Charpin and Mike Zieve for their comments on the preliminary version
of this paper.


\begin{thebibliography}{99}

\bibitem{ball}
S.~Ball, 
\newblock The number of directions determined by a function over a finite field,
{\em J.\,Combin.\,Theory Ser. A}, 104, pp. 341--350 (2003).


\bibitem{charpin-kyureg-gpp}
P.~Charpin and G.~Kyureghyan,
\newblock When does $F(X)+ \gamma Tr(H(X))$ permute  ${\bf F}_{p^n}$?
\newblock submitted, 2008.

\bibitem{charpin-kyureg-bpp}
P.~Charpin and G.~Kyureghyan,
\newblock On a class of permutation polynomials over ${\bf F}_{2^n}$,
\newblock In {\em SETA 2008}, LNCS 5203, pp. 368-376,  Springer-Verlag (2008).

\bibitem{edel-pott}
Y.~Edel and A.~Pott,
A new almost perfect nonlinear function which is not quadratic,
{\em Adv. in Math. of Communications} 3(1), pp.~59-81 (2009).

\bibitem{evans-greene-niederreiter}
R.~J.~Evans, J.~Greene, and H.~Niederreiter,
\newblock Linearized polynomials and permutation polynomials of finite fields,
\newblock {\em Michigan Math. J.}, 39(3), pp.~405--413 (1992).



\bibitem{lai}
X.~Lai, {Additive and linear structures of cryptographic functions},
 {\em Proc. of FSE}, LNCS 1008, pp.~75-85 (1995).

\bibitem{lidl-niederreiter}
R.~Lidl and H.~Niederreiter, {\em Finite Fields}, Encyclopedia of
 Mathematics and its Applications 20. 


\bibitem{marcos}
 J.~E.~Marcos,
\newblock Specific permutation polynomials over finite fields,
\newblock {\em Finite Fileds and their Applications}, in press (2009).



\bibitem{nied-rob}
 H.~Niederreiter and K.H.~Robinson,
\newblock Complete mappings of finite fields,
\newblock {\em J. Austral. Math. Soc.} (Series A) 33, pp.~197-212 (1982).


\bibitem{redei}
L.~R\'edei, \textit{L\"uckenhafte Polynome \"uber endlichen K\"orpern}, Birkh\"auser Verlag, Basel (1970).

\bibitem{zieve}
M.~E.~Zieve, \textit{Classes of Permutation Polynomials Based on Cyclotomy and an Additive Analogue}, arXiv:0810.2830v1.


\end{thebibliography}
\end{document}